\documentclass[12pt,reqno]{amsart}

\usepackage{amsmath, amssymb, amsthm}
\usepackage{enumerate}

\usepackage[margin=1.28in]{geometry}

\begin{document}

 \bibliographystyle{plain}
 \newtheorem{theorem}{Theorem}
 \newtheorem{lemma}[theorem]{Lemma}
 \newtheorem{corollary}[theorem]{Corollary}
 \newtheorem{problem}[theorem]{Problem}
 \newtheorem{conjecture}[theorem]{Conjecture}
 \newtheorem{definition}[theorem]{Definition}
 \newtheorem{prop}[theorem]{Proposition}
 \numberwithin{equation}{section}
 \numberwithin{theorem}{section}

 \newcommand{\mo}{~\mathrm{mod}~}
 \newcommand{\mc}{\mathcal}
 \newcommand{\rar}{\rightarrow}
 \newcommand{\Rar}{\Rightarrow}
 \newcommand{\lar}{\leftarrow}
 \newcommand{\lrar}{\leftrightarrow}
 \newcommand{\Lrar}{\Leftrightarrow}
 \newcommand{\zpz}{\mathbb{Z}/p\mathbb{Z}}
 \newcommand{\mbb}{\mathbb}
 \newcommand{\B}{\mc{B}}
 \newcommand{\cc}{\mc{C}}
 \newcommand{\D}{\mc{D}}
 \newcommand{\E}{\mc{E}}
 \newcommand{\F}{\mathbb{F}}
 \newcommand{\G}{\mc{G}}
  \newcommand{\ZG}{\Z (G)}
 \newcommand{\FN}{\F_n}
 \newcommand{\I}{\mc{I}}
 \newcommand{\J}{\mc{J}}
 \newcommand{\M}{\mc{M}}
 \newcommand{\nn}{\mc{N}}
 \newcommand{\qq}{\mc{Q}}
 \newcommand{\PP}{\mc{P}}
 \newcommand{\U}{\mc{U}}
 \newcommand{\X}{\mc{X}}
 \newcommand{\Y}{\mc{Y}}
 \newcommand{\itQ}{\mc{Q}}
 \newcommand{\sgn}{\mathrm{sgn}}
 \newcommand{\C}{\mathbb{C}}
 \newcommand{\R}{\mathbb{R}}
 \newcommand{\T}{\mathbb{T}}
 \newcommand{\N}{\mathbb{N}}
 \newcommand{\Q}{\mathbb{Q}}
 \newcommand{\Z}{\mathbb{Z}}
 \newcommand{\A}{\mathcal{A}}
 \newcommand{\ff}{\mathfrak F}
 \newcommand{\fb}{f_{\beta}}
 \newcommand{\fg}{f_{\gamma}}
 \newcommand{\gb}{g_{\beta}}
 \newcommand{\vphi}{\varphi}
 \newcommand{\whXq}{\widehat{X}_q(0)}
 \newcommand{\Xnn}{g_{n,N}}
 \newcommand{\lf}{\left\lfloor}
 \newcommand{\rf}{\right\rfloor}
 \newcommand{\lQx}{L_Q(x)}
 \newcommand{\lQQ}{\frac{\lQx}{Q}}
 \newcommand{\rQx}{R_Q(x)}
 \newcommand{\rQQ}{\frac{\rQx}{Q}}
 \newcommand{\elQ}{\ell_Q(\alpha )}
 \newcommand{\oa}{\overline{a}}
 \newcommand{\oI}{\overline{I}}
 \newcommand{\dx}{\text{\rm d}x}
 \newcommand{\dy}{\text{\rm d}y}
\newcommand{\cal}[1]{\mathcal{#1}}
\newcommand{\cH}{{\cal H}}
\newcommand{\diam}{\operatorname{diam}}
\newcommand{\bx}{\mathbf{x}}
\newcommand{\Ps}{\varphi}


\parskip=0.5ex

\title[Linear repetitivity and subadditive ergodic theorems]{Linear repetitivity and\\ subadditive ergodic theorems\\for cut and project sets}
\author[Haynes, Koivusalo, Walton]{Alan~Haynes,~
Henna~Koivusalo,~
James~Walton}
\thanks{Research supported by EPSRC grants EP/L001462, EP/J00149X, EP/M023540.\\ \phantom{aa}HK also gratefully acknowledges the support of Osk. Huttunen foundation}

\allowdisplaybreaks

\begin{abstract}
For the development of a mathematical theory which can be used to rigorously investigate physical properties of quasicrystals, it is necessary to understand regularity of patterns in special classes of aperiodic point sets in Euclidean space. In one dimension, prototypical mathematical models for quasicrystals are provided by Sturmian sequences and by point sets generated by primitive substitution rules. Regularity properties of such sets are well understood, thanks mostly to well known results by Morse and Hedlund, and physicists have used this understanding to study one dimensional random Schr\"{o}dinger operators and lattice gas models. A key fact which plays an important role in these problems is the existence of a subadditive ergodic theorem, which is guaranteed when the corresponding point set is linearly repetitive.\vspace*{.05in}

In this paper we extend the one dimensional model to cut and project sets, which generalize Sturmian sequences in higher dimensions, and which are frequently used in mathematical and physical literature as models for higher dimensional quasicrystals. By using a combination of algebraic, geometric, and dynamical techniques, together with input from higher dimensional Diophantine approximation, we give a complete characterization of all linearly repetitive cut and project sets with cubical windows. We also prove that these are precisely the collection of such sets which satisfy subadditive ergodic theorems. The results are explicit enough to allow us to apply them to known classical models, and to construct linearly repetitive cut and project sets in all pairs of dimensions and codimensions in which they exist.
\end{abstract}

\maketitle

\section{Introduction}

A Delone set $Y\subseteq \R^d$ is {\bf linearly repetitive (LR)} if there exists a constant $C>0$ such that, for any $r\ge 1$, every patch of size $r$ in $Y$ occurs in every ball of diameter $Cr$ in $\R^d$. This concept was studied by F.~Durand (for the special case of subshifts) in \cite{Dura1998,Dura2000}, and by Lagarias and Pleasants in \cite{LagaPlea2003}, and it has since been explored by many authors (e.g. \cite{AlisCoro2011,AlisCoro2013,Besb2008,BesbBoshLenz2013,CortDuraPeti2010,DamaLenz2001}). It was shown in \cite{LagaPlea2003} that linear repetitivity guarantees the existence of strict uniform patch frequencies or, equivalently, strict ergodicity of the associated dynamical system on the hull of the point set. The fact that the Dirac comb of a Delone set in $\R^d$ is a translation bounded measure, together with the existence of strict uniform patch frequencies, implies the existence of a unique autocorrelation measure which is also a translation bounded measure (and hence a tempered distribution). This in turn implies that the associated diffraction measure (the Fourier transform of the autocorrelation) is a positive, translation bounded measure. For this reason, aperiodic LR Delone sets are a common source of examples of mathematical models for quasicrystals, and they are sometimes referred to as `perfectly ordered quasicrystals'.

In the direction of potential physical applications, several authors have studied the random Schr\"{o}dinger operator and lattice gas models on one dimensional quasicrystals, as modeled either by LR Sturmian sequences or point sets constructed using primitive substitution rules \cite{BellIochScopTest1989,IochTest1991,Suto1987}. For both of these applications it is necessary to establish the validity of a uniform subadditive ergodic theorem (see below for definitions). In the case of point sets constructed using primitive substitutions (in fact, in any dimension), such theorems were established by Geerse and Hof in \cite{GeerHof1991}. For LR Sturmian sequences they were established by Lenz in \cite{Lenz2003}. The results of our paper are relevant in this context because of the fact, proved by Damanik and Lenz in \cite{DamaLenz2001}, that linear repetitivity implies the existence of a uniform subadditive ergodic theorem. In fact, one of our main results below (Theorem \ref{thm.PQSuff}) shows that, for a natural class of cut and project sets which are higher dimensional generalizations Sturmian sequences, LR is equivalent to the existence of a subadditive ergodic theorem. The characterization which we provide in Theorem \ref{thm.LRClassification} thus opens the door for the study of physical properties of a rich collection of higher dimensional point sets, a project which is the focus of our current research on this topic.

In \cite[Problem 8.3]{LagaPlea2003}, motivated by many of the connections which we have described, Lagarias and Pleasants asked for a characterization of LR cut and project sets. For $2$ to $1$ cut and project sets with appropriately chosen windows, such a characterization can be obtained using classical results of Morse and Hedlund \cite{MorsHedl1938,MorsHedl1940}. Morse and Hedlund showed that there is a natural bijective correspondence between repetitive Sturmian words and return times, to specially chosen intervals, of orbits of irrational rotations on $\R/\Z$. The latter are precisely given by $2$ to $1$ cut and project sets. Under this correspondence, a Sturmian word will be LR if and only if the irrational rotation which describes it is determined by a badly approximable real number (i.e. a real number whose continued fraction partial quotients are uniformly bounded). The proof of this fact relies on ideas from Diophantine approximation that are closely connected to the theory of continued fractions. In fact the theory is so robust that, in the special case of Sturmian sequences, one can even derive an exact formula for the repetitivity function (see \cite[Theorem 11]{AlesBert1998} and \cite[p.2]{MorsHedl1940}).

For higher dimensional cut and project sets we immediately encounter serious difficulties with generalizing the above mentioned results. One problem is that the underlying dynamical systems which are used for pattern recognition in these sets are, in general, higher rank actions on higher dimensional tori which are more complicated to understand. Another is that there is no known algorithm which can do for higher dimensional cut and project sets what the simple continued fraction algorithm does for Sturmian sequences. Nevertheless, by using a combination of algebraic, geometric, and dynamical tools, together with input from the higher dimensional theory of Diophantine approximation, we are able to obtain the following theorem, which can be seen as a generalization of the classification of Morse and Hedlund, to arbitrary dimensions.
\begin{theorem}\label{thm.LRClassification}
A $k$ to $d$ cubical cut and project set defined by linear forms $\{L_i\}_{i=1}^{k-d}$ is LR if and only if
\begin{itemize}
  \item[(LR1)]  The sum of the ranks of the kernels of the maps $\mc{L}_i:\Z^d\rar\R/\Z$ defined by
  \[\mc{L}_i(n)=L_i(n)~\mathrm{mod}~1\] is equal to $d(k-d-1)$, and\vspace*{.1in}
  \item[(LR2)] Each $L_i$ is relatively badly approximable.
\end{itemize}
\end{theorem}
In the statement of this theorem, a {\bf cubical cut and project set} is one which is regular, aperiodic, and totally irrational, and which is defined using a strip whose intersection with $\R^{k-d}$ is the unit cube $[0,1)^{k-d}$. Precise definitions of these terms are given below.


After proving Theorem \ref{thm.LRClassification} we will explore the relationship between cubical cut and project sets and {\bf canonical cut and project sets} (formed with a window which is the projection of the unit cube in $\R^k$ to the internal space). It will turn out that in many cases which are commonly cited in the literature, the results we state for cubical cut and project sets also apply to canonical ones. However, what is possibly more interesting is that there are examples of LR cubical cut and project sets which are no longer LR when their windows are replaced by canonical ones.

We give examples later in the paper which seem to indicate that there are two potential sources for this type of behavior. The first is geometric, and arises in the situation when at least two of the linear forms defining the physical space have co-kernels with different ranks. The second (which can occur even in the absence of the geometric situation just described) is Diophantine, and is related to the fact that any number can be written as a product of two badly approximable numbers (this follows from continued fraction Cantor set arguments, see \cite{Hall1947}). The complete statements of our results about canonical cut and project sets can be found in Section \ref{sec.Canonical}, but we summarize what has just been mentioned in the following theorem.
\begin{theorem}\label{thm.Cub<->Can}
If $Y$ is a cubical cut and project set which is not LR, then the cut and project set formed from the same data as $Y$, but with the canonical window in place of the cubical one, is also not LR. However, the converse of this statement is not true, in general.
\end{theorem}
It should be pointed out that most of the canonical cut and project sets of specific interest in the literature arise from subspaces defined by linear forms with coefficients in a fixed algebraic number field. In such a case the Diophantine behavior alluded to at the end of the paragraph before Theorem \ref{thm.Cub<->Can} cannot occur. To illustrate this point, in Section \ref{sec.Examples} we will briefly explain how to prove that Penrose and Ammann-Beenker tilings are LR. This in itself is not a new result, and in fact it is fairly obvious from descriptions of these tilings using substitution rules. What is new is that our proof uses only their descriptions as cut and project sets.

Let us now return to discuss the content of Theorem \ref{thm.LRClassification}. In the statement of the theorem, condition (LR1) is necessary and sufficient for $Y$ to have minimal patch complexity. For comparison, in \cite[Section 5]{Juli2010} it was shown that minimal patch complexity is a necessary and sufficient condition for the \v{C}ech cohomology (with rational coefficients) of the associated tiling space to be finitely generated. Many of the calculations in the first part of our proof share a common thread with those that arise in the calculation of the cohomology groups. This point of view eventually leads us to equations \eqref{eqn.LSpacing1}-\eqref{eqn.CSpacing1}, which form the crux of our argument, allowing us to move directly in our proof to a position where we can apply condition (LR2). It should be pointed out that, without the group theoretic arguments that give us the rigid structure imposed by these equations, the proof would fall apart.

Condition (LR2) is a Diophantine condition, which places a strong restriction on how well the subspace defining $Y$ can be approximated by rationals. A linear form is {\bf relatively badly approximable} if it is badly approximable when restricted to rational subspaces complementary to its kernel. In the next section we will explain this in more detail and prove that it is a well defined property. Note that in the special case when $k-d=1,$ condition (LR1) is automatically satisfied, and condition (LR2) requires the linear form defining $Y$ to be badly approximable in the usual sense. This observation, together with Corollary \ref{cor.MLDCubCanon}, immediately implies the following complete characterization of both cubical and canonical cut and project sets in codimension one.
\begin{corollary}\label{cor.LRCodim1}
A $k$ to $k-1$ cubical or canonical cut and project set defined by a linear form $L$ is LR if and only if $L$ is badly approximable.
\end{corollary}


Now we turn to questions more closely related to dynamical properties of cut and project sets. In \cite{BesbBoshLenz2013}, Besbes, Boshernitzan, and Lenz proved that an aperiodic Delone set is LR if and only if it satisfies regularity conditions which they call (PQ) and (U). One consequence of our results is that, for cubical cut and project sets, condition (U) can be omitted from this criteria. In other words, for cubical cut and project sets, condition (PQ) by itself is a necessary and sufficient condition for LR. Results of this type for analogous problems in symbolic dynamics were previously discovered by other authors, and they appear to have been anticipated to hold in greater generality (see \cite[Remark 5]{BesbBoshLenz2013}). These discoveries are of particular interest because of their relevance to the validity of subadditive ergodic theorems. In Section \ref{sec.ThmPQ} we will use Theorem \ref{thm.LRClassification}, together with the results of \cite{BesbBoshLenz2013}, to deduce the following result.
\begin{theorem}\label{thm.PQSuff}
A cubical cut and project set is LR if and only if it satisfies condition (PQ) from \cite{BesbBoshLenz2013}. Consequently, such a set satisfies a subadditive ergodic theorem if and only if it is LR.
\end{theorem}
To begin to make the above statements more precise, let us now give some definitions. Let $E$ be a $d$-dimensional subspace of $\R^k$, and $F_\pi\subseteq\R^k$ a subspace complementary to $E$ (these are referred to as the {\bf physical space} and {\bf internal space}, respectively). Write $\pi$ for the projection onto $E$ with respect to the decomposition $\R^k=E+F_\pi$. Choose a set $\mc{W}_\pi\subseteq F_\pi$, and define $\mc{S}=\mc{W}_\pi+E$. The set $\mc{W}_\pi$ is referred to as the {\bf window}, and $\mc{S}$ as the {\bf strip}. For each $s\in\R^k/\Z^k,$ we define the {\bf cut and project set} $Y_s\subseteq E$ by
\[Y_s=\pi(\mc{S}\cap(\Z^k+s)).\]
In this situation we refer to $Y_s$ as a {\bf $\mathbf{k}$ to $\mathbf{d}$ cut and project set}.

We adopt the conventional assumption that $\pi|_{\Z^k}$ is injective. We also assume in much of what follows that $E$ is a {\bf totally irrational} subspace of $\R^k$, which means that the canonical projection of $E$ into $\R^k/\Z^k$ is dense. This assumption guarantees that the natural linear $\R^d$ action of $E$ on $\R^k/\Z^k$ is uniquely ergodic.


For the problem of studying linear repetitivity, the $s$ in the definition of $Y_s$ plays only a minor role. If we restrict our attention to points $s$ for which $\Z^k+s$ does not intersect the boundary of $\mc{S}$ (these are called {\bf regular} points and the corresponding sets $Y_s$ are called {\bf regular} cut and project sets) then, as long as $E$ is totally irrational, the sets of finite patches in $Y_s$ do not depend on the choice of $s$. In particular, the property of being LR does not depend on the choice of $s$, as long as $s$ is taken to be a regular point. On the other hand, for points $s$ which are not regular, the cut and project set $Y_s$ may contain `additional' patches coming from points on the boundary, which will make it non-repetitive, and therefore not LR. For this reason, {\it we will always assume that $s$ is taken to be a regular point, and we will simplify our notation by writing $Y$ instead of $Y_s$}.

As a point of reference, when allowing $E$ to vary, we also make use of the fixed subspace $F_\rho=\{0\} \times \R^{k-d}\subseteq \R^k$, and we define $\rho:\R^k\rar E$ and $\rho^*:\R^k\rar F_\rho$ to be the projections onto $E$ and $F_\rho$ with respect to the decomposition $\R^k=E+F_\rho$ (assuming, with little loss of generality, that $\R^k$ does in fact decompose in this way). Our notational use of $\pi$ and $\rho$ is intended to be suggestive of the fact that $F_\pi$ is the subspace which gives the {\em projection} defining $Y$ (hence the letter $\pi$), while $F_\rho$ is the subspace with which we {\em reference} $E$ (hence the letter $\rho$). We write $\mc{W}=\mc{S}\cap F_\rho$, and for convenience we also refer to this set as the {\bf window} defining $Y$. This slight ambiguity should not cause any confusion in the arguments below.

As already mentioned, in much of this paper we will focus our attention on the situation where $\mc{W}$ is taken to be a {\bf cubical window}, given by
\begin{equation}\label{eqn.SquareWindow}
\mc{W}=\left\{\sum_{i=d+1}^{k}t_ie_i:0\le t_i<1\right\}.
\end{equation}
In Section \ref{sec.Canonical} we will also consider the case when $\mc{W}$ is taken to be a {\bf canonical window}, which is defined to be the image under $\rho^*$ of the unit cube in $\R^k$.

For any cut and project set, the collection of points $x\in E$ with the property that $Y+x=Y$ forms a group, called the {\bf group of periods} of $Y$. We say that $Y$ is {\bf aperiodic} if the group of periods is $\{0\}$. Finally, as mentioned in the introduction, we say that $Y$ is a {\bf cubical} (resp. {\bf canonical}) {\bf cut and project set} if it is regular, totally irrational, and aperiodic, and if $\mc{W}$ is a cubical (resp. canonical) window.

Without loss of generality, by permuting the standard basis vectors if necessary, we will assume that $E$ can be written as
\begin{equation*}
E = \{(x,L(x)): x \in \R^d\},
\end{equation*}
where $L: \R^d \to \R^{k-d}$ is a linear function. For each $1\le i\le k-d$, we define the linear form $L_i:\R^d \to \R$ by
\[L_i(x) = L(x)_{i} =\sum_{j=1}^d \alpha_{ij} x_j,\]
and we use the points
$\{\alpha_{ij}\}\in\R^{d(k-d)}$ to parametrize the choice of $E$.

Our proof of Theorem \ref{thm.LRClassification} gives an explicit correspondence between the collection of $k$ to $d$ LR cubical cut and project sets, and the Cartesian product of the following two sets:
\begin{itemize}
\item[(S1)] The set of all $(k-d)$-tuples $(L_1,\ldots ,L_{k-d})$, where each $L_i$ is a badly approximable linear form in $m_i\ge 1$ variables, with the integers $m_i$ satisfying $m_1+\cdots +m_{k-d}=d,$ and\vspace*{.1in}
\item[(S2)] The set of all $d\times d$ integer matrices with non-zero determinant.
\end{itemize}
The fact that the set (S1) is empty when $d<k/2$ implies that, for this range of $k$ and $d$ values, there are no LR cubical (or canonical) cut and project sets. On the other hand, for $d\ge k/2$, there are uncountably many, as implied by the following corollary to Theorem \ref{thm.LRClassification}.
\begin{corollary}\label{cor.HausDim}
For $d<k/2$, there are no LR cubical cut and project sets. For $d\ge k/2$, the collection of $\{\alpha_{ij}\}\in\R^{d(k-d)}$ which define LR cubical cut and project sets is a set with Lebesgue measure $0$ and Hausdorff dimension $d$. Furthermore, these statements also apply to canonical cut and project sets.
\end{corollary}
This corollary will be proved in Section \ref{sec.CorHD}. We will also show in Section \ref{sec.Examples} how to exhibit specific examples of LR cut and project sets, for any choice of $d\ge k/2$.\vspace*{.1in}

\noindent {\it Acknowledgements:} Thank you to Oberwolfach for hosting us during the Arbeitsgemeinschaft on Mathematical Quasicrystals, together with a group of people whose input made a significant impact on the final version of this paper. Also, thank you to Franz G\"{a}hler and Uwe Grimm for important observations which helped us to gain some clarity related to the presentation of Section \ref{sec.Canonical}. Finally, we are grateful to Antoine Julien for helpful comments, and for pointing out an error in a previous version of our manuscript.

\section{Definitions and preliminary results}
\subsection{Summary of notation}\label{sec.Notation} For sets $A$ and $B$, the notation $A\times B$ denotes the Cartesian product. If $A$ and $B$ are subsets of the same Abelian group, then $A+B$ denotes the collection of all elements of the form $a+b$ with $a\in A$ and $b\in B$. If $A$ and $B$ are any two Abelian groups then $A\oplus B$ denotes their direct sum.

For $x\in\R,~\{x\}$ denotes the fractional part of $x$ and
$\| x \|$ denotes the distance from $x$ to the nearest integer. For
$x\in\R^m$, we set
$|x|=\max\{|x_1|,\ldots ,|x_m|\}$ and
$\|x\|=\max\{\|x_1\|,\ldots ,\|x_m\|\}.$ We use the symbols $\ll, \gg,$ and $\asymp$ for the standard Vinogradov and asymptotic notation.

The subspaces $E, F_\pi,$ and $F_\rho$, maps $\pi,\rho,$ and $\rho^*$, and sets $\mc{W},\mc {S},$ and $Y$ are defined in the previous section.

Finally, for $y\in Y$ we write $\tilde{y}$ for the point in $\mc{S}\cap\Z^k$ which satisfies $\pi (\tilde{y})=y$. Since $\pi|_{\Z^k}$ is injective, this point is well defined.

\subsection{Results from Diophantine approximation}\label{sec.DiophApp}
Let
$L:\R^d \to \R^{k-d}$ be a linear map given by a matrix with entries
$\{\alpha_{ij}\} \in \R^{d(k-d)}$. For any $N\in\N,$ there exists an $n\in\Z^d$
with $|n|\le N$ and
\begin{equation}\label{eqn.DirichLinForms}
\| L(n) \| \le\frac{1}{N^{d/(k-d)}}.
\end{equation}
This is a multidimensional analogue of Dirichlet's Theorem, which follows from a straightforward application of the pigeonhole principle. We are interested in having an inhomogeneous
version of this result, requiring the values taken by
$\|L(n)-\gamma\|$ to be small, for all choices of
$\gamma\in\R^{k-d}$. For this we will use the following `transference theorem,' a proof of which can be found in \cite[Chapter
V, Section 4]{Cass1957}.
\begin{theorem}\cite[Chapter V, Theorem
VIII]{Cass1957}.\label{thm.CasselsThm}
Given a linear map $L$ as above, the following statements are equivalent:
\begin{itemize}
  \item[(T1)] There exists a constant $C_1>0$ such that
  \[\|L(n)\|\ge \frac{C_1}{|n|^{d/(k-d)}},\]
  for all $n\in\Z^d\setminus\{0\}$.\vspace*{.1in}
  \item[(T2)] There exists a constant $C_2>0$ such that, for all $\gamma\in\R^{k-d},$ the inequalities
  \[\|L(n)-\gamma\|\le \frac{C_2}{N^{d/(k-d)}},\quad |n|\le N,\]
  are soluble, for all $N\ge 1$, with $n\in\Z^d$.
\end{itemize}
\end{theorem}
Next, with a view towards applying this theorem, let $\mathcal{B}_{d,k-d}$
denote the collection of numbers $\alpha\in\R^{d(k-d)}$ with the
property that there exists a constant $C=C(\alpha)>0$ such that, for
all nonzero integer vectors $n\in\Z^d$,
\[\|L(n)\|\ge\frac{C}{|n|^{d/(k-d)}}.\]
The Khintchine-Groshev Theorem (see \cite{BereDickVela2006} for a detailed statement and proof) implies that the Lebesgue measure of $\mathcal{B}_{d,k-d}$ is
$0$. However in terms of Hausdorff dimension these sets are large. It
is a classical result of Jarnik that $\dim \mc{B}_{1,1}=1$, and this
was extended by Wolfgang Schmidt, who showed in \cite[Theorem
2]{Schm1969} that, for any choices of $1\le d<k,$
\[\dim\mc{B}_{d,k-d}=d(k-d).\]

Finally, we introduce the definition of relatively badly approximable linear forms. As mentioned in the introduction, these are linear forms that are badly approximable when restricted to rational subspaces complementary to their kernels. To be precise, suppose that $L:\R^d\rar\R$ is a single linear form in $d$ variables, and define $\mc{L}:\Z^d\rar\R/\Z$ by $\mc{L}(n)=L(n)~\mathrm{mod}~1$. Let $S\leqslant\Z^d$ be the kernel of $\mc{L}$, and write $r=\mathrm{rk}(S)$ and $m=d-r$. We say that $L$ is {\bf relatively badly approximable} if $m>0$ and if there exists a constant $C>0$ and a group $\Lambda\leqslant\Z^d$ of rank $m$, with $\Lambda\cap S=\{0\}$ and
\[\|\mc{L}(\lambda)\|\ge\frac{C}{|\lambda|^{m}}\quad\text{for all}\quad\lambda\in\Lambda\setminus\{0\}.\]
Now suppose that $L$ is relatively badly approximable and let $\Lambda$ be a group satisfying the condition in the definition. Let $F\subseteq\Z^d$ be a complete set of coset representatives for $\Z^d/(\Lambda+S)$. We have the following lemma.
\begin{lemma}\label{lem.RelBad1}
  Suppose that $L$ is relatively badly approximable, with $\Lambda$ and $F$ as above. Then there exists a constant $C'>0$ such that, for any $\lambda\in\Lambda$ and $f\in F$, with $\mc{L}(\lambda+f)\not=0$, we have that
  \[\|\mc{L}(\lambda+f)\|\ge\frac{C'}{1+|\lambda|^{m}}.\]
\end{lemma}
\begin{proof}
  Any element of $F$ has finite order in $\Z^d/(\Lambda+S)$. Therefore, for each $f\in F$  there is a positive integer $u_f$, and elements $\lambda_f\in\Lambda$ and $s_f\in S$, for which
  \[f=\frac{\lambda_f+s_f}{u_f}.\]
  If $\mc{L}(\lambda+f)\not= 0$ then either $\lambda+f=s_f/u_f\not=0$, or $\lambda+u_f^{-1}\lambda_f\not=0$. The first case only pertains to finitely many possibilities, and in the second case we have that
  \begin{align*}
    \|\mc{L}(\lambda+f)\|&\ge u_f^{-1}\cdot\|\mc{L}(u_f\lambda+\lambda_f+s_f)\|\\
    &=u_f^{-1}\cdot\|\mc{L}(u_f\lambda+\lambda_f)\|\\
    &\ge\frac{C}{u_f|u_f\lambda+\lambda_f|^m}.
  \end{align*}
  Therefore, replacing $C$ by an appropriate constant $C'>0$, and using the fact that $F$ is finite, finishes the proof.
\end{proof}
We can also deduce that if $L$ is relatively badly approximable, then the group $\Lambda$ in the definition may be replaced by any group $\Lambda'\leqslant\Z^d$ which is complementary to $S$. This is the content of the following lemma.
\begin{lemma}\label{lem.RelBad2}
Suppose that $L$ is relatively badly approximable. Then, for any group $\Lambda'\leqslant\Z^d$ of rank $m$, with $\Lambda'\cap S=\{0\},$ there exists a constant $C'>0$ such that
\[\|\mc{L}(\lambda')\|\ge\frac{C'}{|\lambda'|^{m}}\quad\text{for all}\quad\lambda'\in\Lambda'\setminus\{0\}.\]
\end{lemma}
\begin{proof}
Let $\Lambda$ be the group in the definition of relatively badly approximable. Choose a basis $v_1,\ldots ,v_m$ for $\Lambda'$, and for each $1\le j\le m$ write
\[v_j=\frac{\lambda_j+s_j}{u_j},\]
with $\lambda_j\in\Lambda, s_j\in S,$ and $u_j\in\N$.

Each $\lambda'\in\Lambda'$ can be written in the form
\[\lambda'=\sum_{j=1}^ma_jv_j,\]
with integers $a_1,\ldots ,a_m$, and we have that
\[\left\|\mc{L}\left(\lambda'\right)\right\|\ge (u_1\cdots u_m)^{-1}\left\|\mc{L}\left(\sum_{j=1}^mb_j\lambda_j\right)\right\|,\]
with $b_j=a_ju_1\cdots u_m/u_j\in\Z$ for each $j$. If the integers $a_j$ are not identically $0$ then, since $\Lambda'\cap S=\{0\}$, it follows that
\[\lambda:=\sum_{j=1}^mb_j\lambda_j\not=0.\]
Using the relatively badly approximable hypothesis gives that
\[\|\mc{L}(\lambda')\|\ge \frac{C}{u_1\cdots u_m\cdot |\lambda|^m}.\]
Finally since $|\lambda|\ll|\lambda'|$, we have that
\[\frac{C}{u_1\cdots u_m\cdot |\lambda|^m}\ge\frac{C'}{|\lambda'|^m},\]
for some constant $C'>0$.
\end{proof}

\subsection{Patterns in cut and project sets}
Let $F_\rho, \rho, \rho^*,$ and $\tilde{y}$ be defined as in Section \ref{sec.Notation}. Assume that we are given a bounded convex set $\Omega\subseteq E$
which contains a neighborhood of $0$ in $E$. Then, for each $r\ge 0$,
define the {\bf patch of size $\mathbf{r}$ at $\mathbf{y}$}, by
\[P(y,r):=\{ y' \in Y: \rho (\tilde{y'}-\tilde y) \in r\Omega\}.\]
In other words, $P(y,r)$ consists of the projections (under $\pi$) to $Y$ of all points of $\mc{S}$ whose first $d$ coordinates are in a certain neighborhood of the first $d$ coordinates of $\tilde y$.

We remark that there are several different definitions of `patches of size $r$' in the literature. For example, from the point of view of $Y$ being contained in $E$, it is more natural to define a patch of size $r$ at $y$ to be the collection of points of $Y$ which lie within distance $r$ of $y$. In \cite{HaynKoivSaduWalt2015} we considered this definition of patch (what we called there type 1 patches), together with the definition that we have given above (type 2 patches). In fact, the two definitions of patches agree except possibly on a constant neighborhood of their boundaries (see \cite[Equation (4.1)]{HaynKoivSaduWalt2015}). Therefore if $Y$ is LR for one definition of `patch of size $r$', it will be LR for the other, and similarly for other reasonable definitions.

For $y_1,y_2\in Y$, we say that $P(y_1,r)$ and
$P(y_2,r)$ are equivalent if
\[P(y_1,r)=P(y_2,r)+y_1-y_2.\] This
defines an equivalence relation on the collection of patches of size $r$. We denote the equivalence class of the patch of size $r$ at $y$ by $\mc{P}(y,r)$.

As indicated in the introduction, a cut and project set $Y$ is {\bf linearly repetitive (LR)} if there is a $C>0$ such that, for every $r>0$, every ball of size $Cr$ in $E$ contains a representative from every equivalence class of patches of size $r$. There are two technical points which will ease our discussion below. First of all, since the points of $Y$ are relatively dense, in the above definition we are free to restrict our attention, without loss of generality (by increasing $C$ if necessary), to balls of size $Cr$ centered at points of $Y$. Secondly, the property of being LR does not depend on the choice of $\Omega$ used to define the patches. This follows from the fact that, if $\Omega'\subseteq E$ is  any other bounded convex set which contains a neighborhood of $0$, then there are dilations of $\Omega'$ which contain, and which are contained in, $\Omega$.

Let $\mc{W}=\mc{S}\cap F_\rho$.  There is a
natural action of $\Z^k$ on $F_\rho$, given by
\[n.w=\rho^*(n)+w = w + (0,n_2-L(n_1)),\]
for $n=(n_1,n_2)\in\Z^k = \Z^d \times
\Z^{k-d}$ and $w\in F_\rho$. For each $r\ge 0$ we define the
{\bf $\mathbf{r}$-singular points} of $\mc{W}$ by
\[\mathrm{sing}(r):=\mc{W}\cap\left((-\rho^{-1}(r\Omega)\cap\Z^k). \partial\mc{W}\right),\]
and the {\bf $\mathbf{r}$-regular points} by
\[\mathrm{reg}(r):=\mc{W}\setminus\mathrm{sing}(r).\]
The following result follows from the proof of \cite[Lemma 3.2]{HaynKoivSaduWalt2015} (see also \cite{Juli2010}).
\begin{lemma}\label{lem.ConnComp}
Suppose that $\mc{W}$ is a parallelotope
generated by integer vectors. For every equivalence class $\mc{P}=\mc{P}(y,r)$, there is a unique connected component $U$ of $\mathrm{reg}(r)$ with the property that, for any $y'\in Y$,
\[\mc{P}(y',r)=\mc{P}(y,r)~\text{ if and only if }~ \rho^*(\tilde{y'})\in U.\]
\end{lemma}
For each equivalence class $\mc{P}=\mc{P}(y,r)$ we define
$\xi_{\mc{P}}$, the {\bf frequency of $\mathbf{\mc{P}}$}, by
\[\xi_{\mc{P}}:=\lim_{R\rar\infty}\frac{\#\{y'\in Y:|y'|\le
R,~\mc{P}(y',r)=\mc{P}(y,r)\}}{\#\{y'\in Y:|y'|\le R\}}.\] It is
not difficult to show that, in our setup, the limit defining
$\xi_{\mc{P}}$ always exists. Lemma \ref{lem.ConnComp}, combined with the Birkhoff Ergodic Theorem, proves that, for totally irrational $E$, the frequencies of equivalence classes of patches are given by the volumes of connected components of $\mathrm{reg}(r)$.
\begin{lemma}\label{lem.FreqVol}
If $E$ is totally irrational then for any $r>0$ and any equivalence class $\mc{P}=\mc{P}(y,r)$, the frequency $\xi_\mc{P}$ is equal to the volume of the connected component $U$ in the statement of Lemma \ref{lem.ConnComp}.
\end{lemma}
This lemma is the full content of \cite[Lemma 3.2]{HaynKoivSaduWalt2015} (this idea is also implicit in \cite{BertVuil2000}).

\subsection{Subadditive ergodic theorems and (PQ)}
In order to facilitate the proof of Theorem \ref{thm.PQSuff}, we briefly gather together some definitions and results from \cite{BesbBoshLenz2013}.

Let $\mc{P}$ be an equivalence class of patches of size $r$, for some $r>0$, and let $B$ be a bounded subset of $E$. Write $\#'_{\mc{P}}B$ for the maximum number of disjoint patches of size $r$ in $B$ which are in the equivalence class $\mc{P}$, and define
\[\nu'(\mc{P})=r^d\cdot\liminf_{|C|\rar\infty}\frac{\#'_\mc{P}C}{|C|},\]
where $C$ runs over all cubes in $E$. We say that $Y$ satisfies {\bf condition (PQ)} if
\[\inf_{\mc{P}}\nu'(\mc{P})>0,\]
where the infimum is taken over all equivalence classes of patches, for all $r>0$.

Let $\mc{B}$ denote the collection of all bounded subsets of $\R^d$, and suppose that $F$ is a function from $\mc{B}$ to $\R$. We say that $F$ is subadditive if, for any disjoint sets $B_1,B_2\in\mc{B}$, we have the inequality
\[F(B_1\cup B_2)\le F(B_1)\cup F(B_2).\]
We say that $F$ is $Y$-invariant if
\[F(B)=F(x+B),\]
whenever $x+(B\cap Y)=(x+B)\cap Y.$ Finally, we say that $Y$ satisfies a {\bf subadditive ergodic theorem} if, for all subadditive $Y$-invariant functions $F$, the limit
\[\lim_{|C|\rar\infty}\frac{F(C)}{|C|}\]
exists. As above, the limit is taken over all cubes $C\subseteq E$.

It follows from \cite[Theorem 1]{BesbBoshLenz2013} that $Y$ satisfies a subadditive ergodic theorem if and only if it satisfies condition (PQ).

\section{Proof of Theorem \ref{thm.LRClassification}}\label{sec.LRClass}
We assume that $\mc{W}$ is given by \eqref{eqn.SquareWindow} and we identify $\mc{W}$ with a subset of $\R^{k-d}$, in the obvious way. Recall that if $Y$ is LR with respect to one convex patch shape $\Omega$, then it is LR with respect to all convex patch shapes. The precise shape $\Omega$ which we will use will be specified later in the proof, but until then everything we say will apply to any fixed choice of such a shape.

For $r>0$ let $c(r)$ denote the number of equivalence classes of patches of size $r$. If $Y$ is LR then there exists a constant $C>0$ such that $c(r)$ is bounded above by $Cr^d$, for all $r>0$. For the first part of the proof of Theorem \ref{thm.LRClassification} we will show that condition (LR1) is necessary and sufficient for a bound of this type to hold.

For each $1\le i\le k-d$, let $S_i\leqslant\Z^d$ denote the kernel of the map $\mc{L}_i$, and let $r_i$ be the rank of $S_i$. Furthermore, for each subset $I\subseteq \{1,\ldots ,k-d\}$ let
\[S_I=\bigcap_{i\in I}S_i,\]
and let $r_I$ be the rank of $S_I$. For convenience, set $S_\emptyset=\Z^d$ and $r_\emptyset=d$. For any pair $I,J\subseteq\{1,\ldots ,k-d\}$, the sum set $S_I+S_J$ is a subgroup of $\Z^d$, and it therefore has rank at most $d$. On the other hand we have that
\[\mathrm{rk}(S_I+S_J)=\mathrm{rk}(S_I)+\mathrm{rk}(S_J)-\mathrm{rk}(S_I\cap S_J),\]
which gives the inequality
\begin{equation}\label{eqn.RankIneq1}
  r_I+r_J\le d+r_{I\cup J}.
\end{equation}
As one application of this inequality we see immediately that
\begin{align}
  r_1+r_2+\cdots +r_{k-d}&\le d+r_{12}+r_3+\cdots +r_{k-d}\nonumber\\
  &\le 2d+r_{123}+r_4+\cdots +r_{k-d}\nonumber\\
  &\phantom{..}\vdots\nonumber\\
  &\le d(k-d-1)+r_{12\ldots (k-d)}\nonumber\\
&= d(k-d-1).\label{eqn.RankIneq2}
\end{align}
The last equality here uses the assumption that $Y$ is aperiodic.

From Lemma \ref{lem.ConnComp}, we know that $c(r)$ is equal to the number of connected components of $\mathrm{reg}(r)$. Let the map $\mc{C}:\Z^{d(k-d)}\rar\mc{W}$ be defined by
\[\mc{C}(n^{(1)},\ldots ,n^{(k-d)})=(\{L_1(n^{(1)})\},\ldots ,\{L_{k-d}(n^{(k-d)})\}),\]
for $n^{(1)},\ldots ,n^{(k-d)}\in\Z^d.$ Identify $\Z^d$ with the set $\mc{Z}=\Z^k\cap \langle e_1,\ldots ,e_d\rangle_\R$, and for each $r>0$ let $\mc{Z}_r\subseteq \Z^d$ be defined by
\[\mc{Z}_r=-\rho^{-1}(r\Omega)\cap\mc{Z}.\]
Since our window $\mc{W}$ is a fundamental domain for the integer lattice in $F_\rho$, there is a one to one correspondence between points of $Y$  and elements of $\mc{Z}$. This correspondence is given explicitly by mapping a point $y\in Y$ to the vector in $\mc{Z}$ given by the first $d$ coordinates of $\tilde{y}$. Also, notice that if $n\in\Z^k$ and $-n.0\in\mc{W}$, then it follows that
\[-n.0=(\{L_1(n_1,\ldots ,n_d)\},\ldots ,\{L_{k-d}(n_1,\ldots ,n_d)\}).\]
These observations together imply that the collection of all vertices of connected components of $\mathrm{reg}(r)$ is precisely the set $\mc{C}(\mc{Z}_r^{k-d}),$ which in turn implies that
\[c(r)\asymp |\mc{C}(\mc{Z}_r^{k-d})|.\]

The values of the function $\mc{C}$ define a natural $\Z^{d(k-d)}$ action on $\mc{W}$. Therefore we may regard the set $\mc{C}(\Z^{d(k-d)})$ as a group, isomorphic to
\[\Z^{d(k-d)}/\mathrm{ker}(\mc{C})\cong \Z^d/S_1\oplus\cdots \oplus \Z^d/S_{k-d}.\]
If (LR1) holds then we have that
\[\mathrm{rk}(\mc{C}(\Z^{d(k-d)}))=~d(k-d)-\sum_{i=1}^{k-d}r_i=~d,\]
and from this it follows that
\[|\mc{C}(\mc{Z}_r^{k-d})|\asymp r^d.\]
On the other hand, if (LR1) does not hold then by \eqref{eqn.RankIneq2} we have that
\[\mathrm{rk}(\mc{C}(\Z^{d(k-d)}))>d,\]
which implies that
\[|\mc{C}(\mc{Z}_r^{k-d})|\gg r^{d+1}.\]
We conclude that $c(r)\ll r^d$ if and only if condition (LR1) holds, so (LR1) is a necessary condition for linear repetitivity.

Next we assume that (LR1) holds and we prove that, under this assumption, condition (LR2) is necessary and sufficient in order for $Y$ to be LR. First of all, suppose that $I$ and $J$ were disjoint, nonempty subsets of $\{1,\ldots ,k-d\}$ for which
\[r_I+r_J<d+r_{I\cup J}.\]
Then, by the same argument used in \eqref{eqn.RankIneq2}, we would have that
\begin{align*}
\sum_{i=1}^{k-d}r_i~\le ~d(k-d-3)+r_{(I\cup J)^c}+r_I+r_J~<~d(k-d-1).
\end{align*}
This clearly contradicts (LR1). Therefore if (LR1) holds then, by \eqref{eqn.RankIneq1}, we have that
\[r_I+r_J=d+r_{I\cup J},\]
whenever $I$ and $J$ are disjoint and nonempty.

For each $1\le i\le k-d$, define $J_i=\{1,\ldots ,k-d\}\setminus \{i\},$ and let $\Lambda_i=S_{J_i}$. Write $m_i=r_{J_i}$ for the rank of $\Lambda_i$. Then, by what was established in the previous paragraph, we have that
\[m_i+r_i=d.\]
If $n$ is any nonzero vector in $\Lambda_i$, then $n$ is in $S_j$ for all $j\not= i$. Since $Y$ is aperiodic, this means that $n\not\in S_i$, which gives that
\[\mathrm{rk}(\Lambda_i+S_i)=m_i+r_i-\mathrm{rk}(\Lambda_i\cap S_i)=d.\]
Furthermore, for any $j\not= i$, the fact that $\Lambda_j\subseteq S_i$ implies that $\Lambda_j\cap \Lambda_i=\{0\}$, so
\[\mathrm{rk}(\Lambda_1+\cdots +\Lambda_{k-d})=\sum_{i=1}^{k-d}\mathrm{rk}(\Lambda_i)=\sum_{i=1}^{k-d}(d-r_i)=d.\]
For each $i$, let $F_i\subseteq\Z^d$ be a complete set of coset representatives for $\Z^d/(\Lambda_i+S_i)$. Also, write $\Lambda=\Lambda_1+\cdots +\Lambda_{k-d}$, and let $F\subseteq\Z^d$ be a complete set of representatives for $\Z^d/\Lambda$. What we have shown so far implies that all of the sets $F_1,\ldots ,F_{k-d},$ and $F$ are finite.

Again thinking of $\Z^d$ as being identified with the set $\mc{Z}$, let
\begin{align*}
\mc{Z}_{r,\Lambda}=\mc{Z}_r\cap\Lambda,\quad \mc{Z}_{r,\Lambda_i}=\mc{Z}_r\cap\Lambda_i,\quad\text{and}\quad\mc{Z}_{r,S_i}=\mc{Z}_r\cap S_i.
\end{align*}
For each $i$, choose a basis $\{v^{(i)}_j\}_{j=1}^{m_i}$ for $\Lambda_i$, and define
\[\Omega_i'=\left\{\sum_{j=1}^{m_i}t_iv^{(i)}_j:-1/2\le t_i< 1/2\right\},\]
and
\[\Omega'=\Omega_1'+\cdots+\Omega_{k-d}',\]
so that $\Omega'$ is a fundamental domain for $\R^d/\Lambda$. We now specify $\Omega$ to be the subset of points in $E$ whose first $d$ coordinates lie in $\Omega'$. In other words,
\[\Omega=E\cap\rho^{-1}(\Omega').\]
Notice that every $n\in\Lambda$ has a unique representation of the form
\[n=\sum_{i=1}^{k-d}\sum_{j=1}^{m_i}a_{ij}v^{(i)}_j,\quad a_{ij}\in\Z.\]
Using this representation, we have that
\begin{align*}
  \mc{L}(n)=\mc{C}((n^{(i)})_{i=1}^{m_i}),
\end{align*}
where, for each $i$, the vector $n^{(i)}\in\Z^d$ is given by
\[n^{(i)}=\sum_{j=1}^{m_i}a_{ij}v^{(i)}_j.\]
This gives a one to one correspondence between elements of $\mc{L}(\Lambda)$ and elements of the set
\[\mc{C}(\Lambda_1\times\cdots \times\Lambda_{k-d})=\mc{L}_1(\Lambda_1)\times\cdots\times\mc{L}_{k-d}(\Lambda_{k-d}).\]
We will combine this observation with the facts that
\begin{equation*}
  \mc{L}(\Z^d)=\mc{L}(\Lambda+F)
\end{equation*}
and
\begin{equation*}
  \mc{C}(\Z^{d(k-d)})=\mc{C}((\Lambda_1+F_1)\times\cdots \times(\Lambda_{k-d}+F_{k-d})),
\end{equation*}
in order to study the spacings between points of the sets $\mc{L}(\mc{Z}_r)$ and $\mc{C}(\mc{Z}_r^{k-d})$.

First of all, it is clear that
\begin{align}\label{eqn.LSpacing1}
\mc{L}(\mc{Z}_{r})~\supseteq~ \mc{L}_1(\mc{Z}_{r,\Lambda_1})\times\cdots\times\mc{L}_{k-d}(\mc{Z}_{r,\Lambda_{k-d}}),
\end{align}
and that
\begin{align}
\mc{C}(\mc{Z}_r^{k-d})~\supseteq~\mc{L}_1(\mc{Z}_{r,\Lambda_1})\times\cdots\times\mc{L}_{k-d}(\mc{Z}_{r,\Lambda_{k-d}}).\label{eqn.CSpacing2}
\end{align}

Since all of the sets $F_1,\ldots ,F_{k-d}$, and $F$ are finite, there is a constant $\kappa>0$ with the property that, for all sufficiently large $r$,
\begin{align*}
\mc{Z}_r~&\subseteq~\mc{Z}_{r+\kappa,\Lambda}+F,\quad\text{and}\\
\mc{Z}_r~&\subseteq~\mc{Z}_{r+\kappa,\Lambda_i}+\mc{Z}_{r+\kappa,S_i}+F_i,
\end{align*}
for each $1\le i\le k-d$. For the second inclusion here we are using the definition of $\Omega$ and the fact that $\Lambda_j\subseteq S_i$ for all $j\not= i$. These inclusions imply that
\begin{align}
 \mc{L}(\mc{Z}_{r})~&\subseteq~\mc{L}(\mc{Z}_{r+\kappa,\Lambda})+\mc{L}(F)\nonumber\\
 &\subseteq~\mc{L}_1(\mc{Z}_{r+\kappa,\Lambda_1}+F)\times\cdots\times\mc{L}_{k-d}(\mc{Z}_{r+\kappa,\Lambda_{k-d}}+F),\label{eqn.LSpacing2}
\end{align}
and that
\begin{align}
\mc{C}(\mc{Z}_r^{k-d})~&\subseteq ~
\mc{C}\left((\mc{Z}_{r+\kappa,\Lambda_1}+F_1)\times\cdots\times(\mc{Z}_{r+\kappa,\Lambda_{k-d}}+F_{k-d})\right)\nonumber\\
&=~
\mc{L}_1(\mc{Z}_{r+\kappa,\Lambda_1}+F_1)\times\cdots\times\mc{L}_{k-d}(\mc{Z}_{r+\kappa,\Lambda_{k-d}}+F_{k-d}).\label{eqn.CSpacing1}
\end{align}
Now we are positioned to make our final arguments.

Suppose first of all that (LR2) holds. Let $U$ be any connected component of $\mathrm{reg}(r)$. Then $U$ is a $(k-d)$-dimensional box, with faces parallel to the coordinate hyperplanes, and with vertices in the set $\mc{C}(\mc{Z}_r^{k-d}).$  Therefore we can write $U$ in the form
\begin{equation}\label{eqn.ConnCompU}
U=\{x\in \mc{W}: \ell_i<x_i<r_i\},
\end{equation}
where for each $i$, the values of $\ell_i$ and $r_i$ are elements of the set $\mc{L}_i(\mc{Z}_r)$. By equation \eqref{eqn.CSpacing1}, together with Lemma \ref{lem.RelBad1}, there is a constant $c_1>0$ such that, for every $i$,
\[r_i-\ell_i\ge\frac{c_1}{r^{m_i}}.\]

Next we will show that there is a constant $c_2>0$ such that, for all sufficiently large $r$, the orbit of every point in $F_\rho/\Z^{k-d}$ under the action of $\mc{Z}_{c_2r}$ intersects every connected component of $\mathrm{reg}(r)$. Then Lemma \ref{lem.ConnComp} will imply that $Y$ is LR. To show that there is such a constant $c_2$, we use \eqref{eqn.LSpacing1} and Theorem \ref{thm.CasselsThm}. Each one of the linear forms $L_i$  is a badly approximable linear form in $m_i$ variables, when restricted to $\Lambda_i$. Therefore, by (T2) of Theorem \ref{thm.CasselsThm}, there is a constant $\eta>0$ with the property that, for all sufficiently large $r$ and for each $i$, the collection of points $\mc{L}_i(\mc{Z}_{c_2r,\Lambda_i})$ is $\eta/(c_2r^{m_i})$-dense in $\R/\Z$. Choosing $c_2>3c_1/\eta$ completes the proof of this part of the theorem, verifying that (LR1) and (LR2) together imply linear repetitivity.

For the final part, suppose that (LR1) holds and (LR2) does not. Then one of the linear forms $L_i$ is not relatively badly approximable, and we assume without loss of generality that it is $L_1$. Let $c_2$ be any positive constant, and consider the collection of points $\mc{L}(\mc{Z}_{c_2r})$. By \eqref{eqn.LSpacing2}, the first coordinates of these points are a subset of \[\mc{L}_1(\mc{Z}_{c_2r+\kappa,\Lambda_1}+F).\]
There are at most $c_2\delta r^{m_1}-1$ points in the latter set, for some constant $\delta$ depending on $\Lambda_1$. Therefore, thinking of the points as being arranged in increasing order in $[0,1)$, there must be two consecutive points which are at least $1/(c_2\delta r^{m_1})$ apart. On the other hand, by \eqref{eqn.CSpacing2} and our hypothesis on $L_1$, we can choose $r$ large enough so that there is a connected component $U$ of $\mathrm{reg}(r)$, given as in \eqref{eqn.ConnCompU}, with
\[r_1-\ell_1<\frac{1}{c_2\delta r^{m_1}}.\]
From these two observations it is clear that there is some point in $F_\rho/\Z^{k-d}$ whose orbit under $\mc{Z}_{c_2r}$ does not intersect $U$. Since $c_2>0$ was arbitrary, this means that $Y$ is not LR. Therefore, (LR1) and (LR2) are necessary conditions for linear repetitivity, and the proof of Theorem \ref{thm.LRClassification} is complete.

\section{Canonical cut and project sets}\label{sec.Canonical}
In this section we turn our attention to canonical cut and project sets. In order to gain a broader perspective on our results, we will use the notion of local derivability for point sets, first introduced in \cite{BaakSchlJarv1991}. Suppose that $Y_1$ and $Y_2$ are two cut and project sets formed from a common physical space $E$, and suppose (without loss of generality for the purposes of all of our results) that $Y_1$ and $Y_2$ are both uniformly discrete and relatively dense. We say that {\bf $\mathbf{Y_1}$ is locally derivable from $\mathbf{Y_2}$} if there exists a constant $c>0$ with the property that, for all $x\in E$ and for all sufficiently large $r$, the equivalence class of the patch of size $r$ centered at $x$ in $Y_2$ uniquely determines the patch of size $r-c$ centered at $x$ in $Y_1$.  There is a minor technical issue here, that $x$ may not belong to $Y_1$ or $Y_2$. However, since $Y_1$ and $Y_2$ are relatively dense, this can be rectified by requiring that $x$ be moved, when necessary in the definition above, to a nearby point of the relevant cut and project set. Finally, we say that $Y_1$ and $Y_2$ are {\bf mutually locally derivable (MLD)} if each set is locally derivable from the other.

The argument in \cite[Appendix]{BaakSchlJarv1991} (see also \cite{Baak2002} and \cite[Remark 7.6]{BaakGrim2013}) provides us with the following characterization of MLD cut and project sets $Y_1$ and $Y_2$ as above.
\begin{lemma}\label{lem.MLDCutAndProj}
  Let $Y_1$ and $Y_2$ be regular, totally irrational, aperiodic $k$ to $d$ cut and project sets, constructed with the same physical and internal spaces and with windows $\mc{W}_1$ and $\mc{W}_2$, respectively. Then $Y_1$ is locally derivable from $Y_2$ if and only if $\mc{W}_1$ is a finite union of sets each of which is a finite intersection of translates of $\mc{W}_2$ (or of its complement), with translations taken from $\rho^*(\Z^k)$.
\end{lemma}
From this lemma we immediately deduce the following result relating cubical and canonical cut and project sets.
\begin{corollary}\label{cor.MLDCubCanon}
Let $Y_1$ be a cubical cut and project set,
and let $Y_2$ be the cut and project set formed from the same data as $Y_1$, but with the canonical window. Then $Y_1$ is locally derivable from $Y_2$. Furthermore, $Y_2$ is locally derivable from $Y_1$ if and only if, for each $1\le i\le d$, the point $\rho^*(e_i)$ lies on a line of the form $\R e_j$, for some $d+1\le j\le k$.
\end{corollary}
If a Delone set $Y_1$ is locally derivable from $Y_2$, and if $Y_2$ is LR, then it follows directly from the definitions that $Y_1$ is also LR. Therefore, Corollary \ref{cor.MLDCubCanon} implies the first part of Theorem \ref{thm.Cub<->Can}, that if a cubical cut and project set is not LR, then neither is the corresponding canonical cut and project set.

The second part of Theorem \ref{thm.Cub<->Can} is not quite as obvious. To understand the issue, note that it is easy, using Theorem \ref{thm.LRClassification} and Corollary \ref{cor.MLDCubCanon}, to come up with cut and project sets $Y_1$ and $Y_2$, as in the statement of the corollary, for which $Y_1$ is LR but $Y_2$ is not locally derivable from $Y_1$. For example, the subspace used to define the Ammann-Beenker tiling in Section \ref{sec.AmmBee} provides us with such sets. On the other hand, as can be seen in the Ammann-Beenker example, in general this does not imply that $Y_2$ is not LR.

We will demonstrate two different constructions for producing cubical cut and project sets which are LR, but for which their canonical counterparts are not. Our first construction is based on elementary geometric considerations.
\begin{lemma}\label{lem.Cub<->Can1}
Suppose that $\alpha_1,\alpha_2,$ and $\beta$ are positive real numbers with $(\alpha_1,\alpha_2)\in\mathcal{B}_{2,1}$ and $\beta\in\mathcal{B}_{1,1}$, and let $E$ be the three dimensional subspace of $\R^5$ defined by
\[E=\{(x,-\alpha_1x_1-\alpha_2x_2-x_3,-\beta x_3):x\in\R^3\}.\]
Then every cubical cut and project set defined using $E$ is LR, but no canonical cut and project sets defined using $E$ are LR.
\end{lemma}
\begin{proof}
Observe, as is implicit in the statement of the lemma, that $E$ is totally irrational, and that any cubical or canonical cut and project set defined using $E$ is aperiodic. In the notation of the proof of Theorem \ref{thm.LRClassification}, the kernels of the linear forms $\mc{L}_1$ and $\mc{L}_2$ have ranks $r_1=1$ and $r_2=2$. Therefore Theorem \ref{thm.LRClassification} allows us to conclude that any cubical cut and project set formed using $E$ is LR.

Let $\mc{W}$ be the canonical window in $F_\rho$. The window is a hexagon with vertices at
\begin{eqnarray*}
e_5,~e_4+(1+\beta)e_5,~(2+\alpha_1+\alpha_2)e_4+(1+\beta)e_5,\\
(2+\alpha_1+\alpha_2)e_4+\beta e_5,~(1+\alpha_1+\alpha_2)e_4,~\text{and}~0.
\end{eqnarray*}
For $r>0$, let us consider the orbit of the line segment from $e_5$ to $e_4+(1+\beta)e_5$, under the action of the collection of integers
\begin{equation}\label{eqn.IntCylinder}
-\rho^{-1}(r\Omega)\cap\Z^k
\end{equation}
used to define $\mathrm{sing}(r)$. By our Diophantine hypotheses on the subspace $E$ (using the transference principles in the same way as we have in the proof of Theorem \ref{thm.LRClassification}), there is a constant $C>0$ with the property that, for all sufficiently large $r$, there is an integer point $n$ in the set \eqref{eqn.IntCylinder} satisfying
\[n.e_5=(y_1,y_2),\]
with
\begin{equation}\label{eqn.yPlacement}
1-\frac{C}{r^2}<y_1<1~\text{ and }~\left|y_2-\frac{1}{2}\right|<\frac{C}{r}.
\end{equation}
Provided $r$ is large enough, the line segment from $n.e_5$ to $n.(e_4+(1+\beta)e_5)$, together with the lines $n.e_5+\R e_4$ and $e_4+\R e_5$, bound a triangle $\mc{T}_r$ with
\begin{equation}\label{eqn.Triangle1}
  |\mc{T}_r|\ll_\beta \frac{1}{r^4}.
\end{equation}

Now we claim that if $y$ and $y'$ are two points in the canonical cut and project set formed using $\mc{W}$, and if $\rho^*(\tilde{y})\in \mc{T}_r$ but $\rho^*(\tilde{y'})\notin \mc{T}_r$, then $P(y,r)$ and $P(y',r)$ are not in the same equivalence class of patches of size $r$. Note that we cannot appeal directly to Lemma \ref{lem.ConnComp} in this case, since the window we are using does not satisfy its hypotheses. In this case we argue directly as follows. Write $y^*=\rho^*(\tilde{y})$ and $y'^*=\rho^*(\tilde{y'})$, suppose that $y^*\in\mc{T}_r$ and $y'^*\not\in\mc{T}_r$, and write $\ell_n$ for the line segment from $n.e_5$ to $n.(e_4+(1+\beta)e_5)$. Consider the following three cases:\vspace*{.1in}

\noindent{\it Case 1:} If $y'^*$ lies in the half-plane above the line containing $\ell_n$ then the point $\tilde{y}-n$ lies in $\mc{S}$, while $\tilde{y'}-n$ does not. Therefore $y-\pi (n)\in P(y,r)$ but $y'-\pi (n)\notin P(y',r)$.\vspace*{.1in}

\noindent{\it Case 2:} If $y'^*$ lies in the half-plane below the line $n.e_5+\R e_4$ then the point $\tilde{y}+n+e_5$ lies in $\mc{S}$, while $\tilde{y'}-n-e_5$ does not. As in the previous case, $y-\pi (n+e_5)\in P(y,r)$ but $y'-\pi (n+e_5)\notin P(y',r)$.\vspace*{.1in}

\noindent{\it Case 3:} If neither Case 1 nor Case 2 applies, then $y'^*$ lies to the right of the line $e_4+\R e_5$, in the cone below the line containing $\ell_n$ and above the line $n.e_5+\R e_4$. It is clear that $\tilde{y}-e_4\notin\mc{S}$, and if $\tilde{y'}-e_4\in\mc{S}$ then the argument is the same as before. Otherwise, if $\tilde{y'}-e_4\notin\mc{S}$, then we would have to have that $y'^*\notin \mc{W}+e_4$. Since we have arranged in \eqref{eqn.yPlacement} for $y_2$ to be close to $1/2$, this would imply that $|y^*-y'^*|\gg 1$. As long as $r$ is sufficiently large, we could then conclude, by the same types of arguments used in Cases 1 and 2, that $P(y,r)$ contains a point which does not appear in $P(y',r)$.\vspace*{.1in}

Since these cases cover all possibilities, we have verified our claim above, that the connected component $\mc{T}_r$ corresponds to a unique equivalence class of patches of size $r$. This, together with \eqref{eqn.Triangle1} and a simple application of the Birkhoff Ergodic Theorem, shows that for all sufficiently large $r$, there are equivalence classes of patches of size $r$ which occur with frequency $\ll r^{-4}$. Since $d=3$, we conclude that canonical cut and project sets formed using $E$ cannot be LR.
\end{proof}
Although we will not attempt to generalize Lemma \ref{lem.Cub<->Can1}, we remark that a similar construction would likely work whenever $d>k-d>1$, to show that some canonical cut and project sets are not LR (and even with the extra requirement that their cubical counterparts are LR). We posit that the analogous conditions necessary to draw this conclusion from the argument just given should be, in the notation of the proof of Theorem \ref{thm.LRClassification}, that there are integers $1\le j<j'\le k-d$ satisfying
\begin{enumerate}
\item[(i)] $r_j\not= r_{j'}$, and\vspace*{.05in}
\item[(ii)] there exists an integer $1\le i\le d$ such that the orthogonal projection (with respect to the standard basis vectors) of the vector $\rho^*(e_i)\in F_\rho$, onto the plane in $F_\rho$ spanned by $e_j$ and $e_{j'}$, does not lie on either of the lines $\R e_j$ or $\R e_{j'}$.
\end{enumerate}
Condition (i) is simply the requirement that the kernels of two of the linear forms defining $E$, considered modulo $1$, have different ranks. The second condition is that there is a `slant' in the window, when projected orthogonally to the $e_je_{j'}$-plane.

Interestingly, there is a different type of behavior which can cause canonical cut and project sets to fail to be LR. This behavior is related to Diophantine approximation, and occurs because of the fact that two subspaces defined by relatively badly approximable linear forms can still intersect in a subspace which is not definable using relatively badly approximable forms. A one dimensional realization of this fact is the famous theorem of Marshall Hall \cite[Theorem 3.2]{Hall1947}, which implies that every non-zero real number can be written as a product of two badly approximable numbers. This provides the basis for the following example with $k=4$ and $d=2$, the smallest possible choices of $k$ and $d$ for which `cubical LR but canonical not' can occur.
\begin{lemma}\label{lem.Cub<->Can2}
Suppose that $\alpha$ and $\beta$ are positive badly approximable real numbers with the property that
\begin{equation}\label{eqn.WellAppHyp}
  \inf_{n\in\N}~n\cdot\left\{\frac{5\alpha\beta}{2}n\right\}=0.
\end{equation}
If $E$ is the two dimensional subspace of $\R^4$ defined by
\[E=\{(x,-(2/5)x_1-\alpha x_2,-\beta(x_1+(5/2)x_2)):x\in\R^2\},\]
then every cubical cut and project set defined using $E$ is LR, but no canonical cut and project sets defined using $E$ are LR.
\end{lemma}
\begin{proof}
First of all we remark that, by a general version of Khintchine's Theorem (see \cite[Theorem 1]{Schm1964}), almost every real number $\gamma$ has the property that
\[\inf_{n\in\N}n\{n\gamma\}=0.\]
For such a $\gamma$, it follows from Hall's Theorem that there are badly approximable $\alpha$ and $\beta$ satisfying $5\alpha\beta/2=\gamma$, and therefore (assuming $\gamma>0$) the hypotheses of the lemma.\vspace*{.05in}

It is easy to see that $E$ is totally irrational and that cubical and canonical cut and project sets formed using $E$ will be aperiodic. By Theorem \ref{thm.LRClassification}, cubical cut and project sets formed using $E$ will be LR.\vspace*{.05in}

The canonical window in $F_\rho$ is an octagon which includes, on its boundary, the line segment from $e_4$ to $(2/5)e_3+(1+\beta)e_4$. Each integer $n\in\Z^4$ acts on this line segment, moving it to a line segment which we denote by $\ell_n$. The initial point of $\ell_n$ is the point
\[((2/5)n_1+\alpha n_2+n_3,\beta n_1+(5\beta/2)n_2+1+n_4)\]
in the $e_3e_4$-plane. For any choice of $n_2,n_3 ,$ and $n_4$, there is a unique choice of $n_1$  with the property that $\ell_n$ intersects the line $e_3+\R e_4$, and it is clear that $|n_1|$ is bounded above by a constant (depending at most on $\alpha$ and $\beta$) times the maximum of $|n_2|,|n_3|,$ and $|n_4|$. The intersection point just described is $e_3+ye_4$, where $y=y(n_2,n_3,n_4)$ is given by
\[y=\frac{-5\alpha\beta}{2}n_2+1+n_4+\frac{5\beta}{2}(1+n_2-n_3).\]
Since \eqref{eqn.WellAppHyp} is satisfied, for any $\epsilon>0$ there is a number $r>1$ and integers $n_2$ and $n_4$ with $|n_2|,|n_4|\le r$, such that
\[\left|\frac{5\alpha\beta}{2}n_2-(1+n_4)+1\right|<\frac{\epsilon}{r}.\]
For such a choice of $n_2$ and $n_4$, and with $n_1$ selected as above, we take $n_3=1+n_2$. Then the line segment $\ell_n$, together with the lines $e_4+\R e_3$ and $e_3+\R e_4$, bound a triangle of area $\ll \epsilon r^{-2}$. Since $\epsilon$ can be taken arbitrarily small, the remainder of the proof follows from the same type of argument as the one used at the end of the proof of Lemma \ref{lem.Cub<->Can1}.
\end{proof}
To bring us to our concluding remarks for this section, we first of all mention that the proof we have just given is somewhat biased towards one particular point of view. In fact, there is a conceptually easier proof (which is instructive in a different way), which we now explain. If we reparameterize the subspace $E$ in the statement of the lemma, with respect to the `reference' subspace generated by $e_3$ and $e_4$, then it is easy to see that condition (LR1) from Theorem \ref{thm.LRClassification} is not satisfied, so the corresponding cubical cut and project sets are not LR. Therefore canonical cut and project sets obtained using $E$ (which of course do not depend on the choice of reference subspace) are also not LR.

This simple consideration leads us to ask the following question:
\begin{problem}
Is it true that a canonical cut and project set will be LR if and only if all of the cubical cut and project sets obtained from taking different parameterizations of $E$, with respect to different orderings of the standard basis vectors, are also LR?
\end{problem}
At the moment we do not know the answer to this question. However, if the answer is yes, it means that Theorem \ref{thm.LRClassification} gives a complete characterization of all canonical as well as cubical cut and project sets. We leave this as an open problem for future research.

\section{Proof of Theorem \ref{thm.PQSuff}}\label{sec.ThmPQ}
In this section we turn to the problems of relating our conditions for LR to those from \cite{BesbBoshLenz2013}, and of determining when subadditive ergodic theorems hold for cut and project sets. Let $Y$ be a cubical cut and project set. It follows from \cite[Theorem 2]{BesbBoshLenz2013} that $Y$ is LR if and only if it satisfies conditions (PQ) and (U) (we will not define condition (U), since it is unnecessary for our purposes). Therefore, if $Y$ is LR then it satisfies condition (PQ) and, by \cite[Theorem 1]{BesbBoshLenz2013}, it satisfies a subadditive ergodic theorem.

In the other direction, suppose that $Y$ satisfies condition (PQ) (equivalently, that $Y$ satisfies a subadditive ergodic theorem). Then the fact that $\nu'(\mc{P})>0$ for all equivalence classes $\mc{P}$ implies that there is a constant $C>0$ with the property that, for any equivalence class $\mc{P}=\mc{P}(y,r)$, we have that
 \begin{equation}\label{eqn.FreqLowBd}
 \xi_\mc{P}\ge\frac{C}{r^d}.
 \end{equation}
It follows, first of all, that the number of patches of size $r$ is $\ll r^d$. By the first part of our proof of Theorem \ref{thm.LRClassification}, this implies that condition (LR1) is satisfied.

Now suppose that condition (LR2) is not satisfied and assume, without loss of generality, that the linear form $L_1$ defining $E$ is not relatively badly approximable.  Then, for any $\epsilon>0$, we can choose $r>0$ so that there is a connected component $U$ of $\mathrm{reg}(r)$, given as in \eqref{eqn.ConnCompU}, with
\[r_1-\ell_1<\frac{\epsilon}{r^{m_1}}.\]
We can also assume that $U$ has been chosen so that
\[r_i-\ell_i\ll\frac{1}{r^{m_i}},\quad\text{for each}~2\le i\le k-d.\]
Then $U$ has volume $\ll\epsilon /r^d$, where the implied constant depends only on $E$. Since $\epsilon$ can be taken arbitrary small, this together with \eqref{eqn.FreqLowBd} and Lemma \ref{lem.FreqVol} implies that (PQ) is not satisfied. Therefore, (PQ) implies (LR1) and (LR2), which, by Theorem \ref{thm.LRClassification}, implies linear repetitivity.

\section{Hausdorff dimension results}\label{sec.CorHD}
In this section we will prove Corollary \ref{cor.HausDim}. Our proof of Theorem \ref{thm.LRClassification} demonstrates how, to each LR cubical cut and project set, we may associate a subgroup $\Lambda\leqslant\Z^d$ of finite index, with decomposition
\[\Lambda=\Lambda_1+\cdots +\Lambda_{k-d},\]
so that each $L_i$ is badly approximable, when viewed as a linear form in $m_i$ variables, restricted to $\Lambda_i$. The first part of Corollary \ref{cor.HausDim} clearly follows from the fact that the integers $m_i\ge 1$ have sum equal to $d$.

In the other direction, suppose that $d\ge k-d$. If we start with $k-d$ positive integers $m_i$, with sum equal to $d$, and a collection of badly approximable linear forms $L_i:\R^{m_i}\rar\R$ then, thinking of
\[\R^d=\R^{m_1}+\cdots +\R^{m_{k-d}},\]
any cubical cut and project set arising from the subspace
\[E=\{(x,L_1(x),\ldots ,L_{k-d}(x)):x\in\R^d\}\]
is LR, by the proof of Theorem \ref{thm.LRClassification}. It follows that the collection of $\{\alpha_{ij}\}\in\R^{d(k-d)}$ which define LR cubical cut and project sets is a countable union (over all allowable choices of $\Lambda_i$ above) of sets of Lebesgue measure $0$ and Hausdorff dimension at most
\[\mathrm{dim}\mc{B}_{m_1,1}+\cdots+\mathrm{dim}\mc{B}_{m_{k-d},1}=m_1+\cdots +m_{k-d}=d.\]
Since the cubical cut and project sets corresponding to $\Lambda_i=\Z^{m_i}$ are all LR, the Hausdorff dimension of this set is equal to $d$.

The part of Corollary \ref{cor.HausDim} about canonical cut and project sets follows from the same arguments just given, together with Corollary \ref{cor.MLDCubCanon}.

\section{Examples}\label{sec.Examples}
\subsection{Explicit examples for all $d\ge k-d$}
For $d\ge k/2$ it is easy to give examples of subspaces $E$ satisfying the hypotheses of Theorem \ref{thm.LRClassification}. Write $d=m_1+\cdots +m_{k-d}$, with positive integers $m_i$, and for each $i$ let $K_i$ be an algebraic number field, of degree $m_i+1$ over $\Q$. Suppose that the numbers $1,\alpha_{i1},\ldots ,\alpha_{im_i}$ form a $\Q$-basis for $K_i$, and define $L_i:\R^{m_i}\rar \R$ to be the linear form with coefficients $\alpha_{i1},\ldots ,\alpha_{im_i}$. Then, using the decomposition $\R^d=\R^{m_1}+\cdots +\R^{m_{k-d}}$, let
\[E=\{(x,L_1(x),\ldots ,L_{k-d}(x)):x\in\R^d\}.\]
The collection of points
\[\{(L_1(n),\ldots ,L_{k-d}(n)):n\in\Z^d\}\]
is dense in $\R^{k-d}/\Z^{k-d},$ and it follows from this that the subspace $E$ is totally irrational. The intersection of the kernels of the corresponding maps $\mc{L}_i$ is $\{0\},$ so any cubical cut and project set formed from $E$ will be aperiodic. Condition (LR1) of Theorem \ref{thm.LRClassification} is clearly satisfied. Furthermore, by a result of Perron \cite{Perr1921}, each of the linear forms $L_i$ is badly approximable. Therefore (LR2) is also satisfied, and any cubical cut and project set formed from $E$ is LR. Furthermore, the hypotheses in the second part of Corollary \ref{cor.MLDCubCanon} are also satisfied, so any canonical cut and project set formed using $E$ is also LR.

\subsection{Ammann-Beenker tilings}\label{sec.AmmBee}
Collections of vertices of Ammann-Beenker tilings can be obtained as canonical cut and project sets, using the two dimensional subspace $E$ of $\R^2$ defined by
\[E=\{(x,L_1(x),L_2(x)):x\in\R^2\},\]
with
\begin{align*}
  L_1(x)=\frac{\sqrt{2}}{2}(x_1+x_2)\quad\text{and}\quad L_2(x)=\frac{\sqrt{2}}{2}(x_1-x_2).
\end{align*}
Although we cannot directly appeal to either Theorem \ref{thm.LRClassification} or Corollary \ref{cor.MLDCubCanon}, we will explain how the machinery we have developed can be used to easily show that these sets are LR. The canonical window $\mc{W}$ in $F_\rho$ is the regular octagon with vertices at
\[\left(\frac{1+\sqrt{2}}{2}\pm\frac{1+\sqrt{2}}{2}~,~\frac{1}{2}\pm\frac{1}{2}\right)~\text{and}~ \left(\frac{1+\sqrt{2}}{2}\pm\frac{1}{2}~,~\frac{1}{2}\pm\frac{1+\sqrt{2}}{2}\right).\]
By the proof of \cite[Lemma 3.1]{HaynKoivSaduWalt2015}, every patch of size $r$ corresponds (in the sense of the statement of Lemma \ref{lem.ConnComp}) to a finite collection of connected components of $\mathrm{reg}(r)$. Therefore to demonstrate that a canonical cut and project set formed using $E$ is LR, it is enough to show that the there is a constant $C>0$ with the property that, for all sufficiently large $r$, the orbit of any regular point $w\in F_\rho$, under the action of the collection of integers
\[\rho^{-1}(Cr\Omega)\cap\Z^k,\]
intersects every connected component of $\mathrm{reg}(r)$.

We claim that every connected component of $\mathrm{reg}(r)$ contains a square with side length $\gg r^{-1}$. This follows from elementary considerations, by writing down the equations of the line segments that form the boundary of $\mc{W}$, considering the action of
\[-\rho^{-1}(r\Omega)\cap\Z^k\]
on these line segments, and then computing all possible intersection points. Since all of our algebraic operations take place in the field $\Q(\sqrt{2})$, it is not difficult to show that every connected component of $\mathrm{reg}(r)$ must contain a right isosceles triangle of side length $\gg r^{-1}$. The claim about squares follows immediately.

 Finally, the linear forms $L_1$ and $L_2$ are relatively badly approximable, and the sum of the ranks of $\mc{L}_1$ and $\mc{L}_2$ is equal to $2$. Therefore our study of the orbits of points towards the end of the proof of Theorem \ref{thm.LRClassification} applies as before, allowing us to conclude that the $Cr$-orbit of any regular point in $F_\rho$ intersects every connected component of $\mathrm{reg}(r)$.


\subsection{Penrose tilings}\label{sec.Penrose}
This example is similar to the previous one, but it also gives an indication of how to apply our techniques in cases when the physical space is not totally irrational. Let $\zeta=\exp(2\pi i/5)$ and let $Y$ be a canonical cut and project set defined using the two dimensional subspace $E$ of $\R^5$ generated by the vectors
\[(1,\mathrm{Re}(\zeta),\mathrm{Re}(\zeta^2),\mathrm{Re}(\zeta^3),\mathrm{Re}(\zeta^4))\]
and
\[(0,\mathrm{Im}(\zeta),\mathrm{Im}(\zeta^2),\mathrm{Im}(\zeta^3),\mathrm{Im}(\zeta^4)).\]
Well known results of de Bruijn \cite{Brui1981} and Robinson \cite{Robi1996} show that the set $Y$ is the image under a linear transformation of the collection of vertices of a Penrose tiling, and in fact that all Penrose tilings can be obtained in a similar way from cut and project sets. The fact that $Y$ is LR can be deduced directly from the definition of the Penrose tiling as a primitive substitution. However, as in the previous example, we will show how to prove this starting from the definition of $Y$ as a cut and project set.

The subspace $E$ is contained in the rational subspace orthogonal to $(1,1,1,1,1)$. In this case Theorem \ref{thm.LRClassification} does not apply directly, but the proof in Section \ref{sec.LRClass} is still robust enough to allow us to draw the desired conclusions. Set
\[\alpha_1=\cos (2\pi/5), \alpha_2=\cos (4\pi/5), \beta_1=\sin (2\pi/5),\text{ and } \beta_2=\sin (4\pi/5),\]
so that
\[E=\{(x, x\alpha_1+y\beta_1, x\alpha_2+y\beta_2, x\alpha_2-y\beta_2, x\alpha_1-y\beta_1):x,y\in\R\}.\]
After making the change of variables $x_1=x$ and $x_2=x\alpha_1+y\beta_1$, we can write $E$ as
\[E=\{(x,L_1(x),L_2(x),L_3(x)):x=(x_1,x_2)\in\R^2\}.\]
The functions $L_i$ are linear forms which (using the fact that $4\alpha_1^2+2\alpha_1-1=0$) are given by
\begin{align*}
L_1(x)&=-x_1+2\alpha_1x_2,\\
L_2(x)&=-2\alpha_1x_1-2\alpha_1x_2,~\text{and}\\
L_3(x)&=2\alpha_1x_1-x_2.
\end{align*}
Write $\mc{L}_i:\Z^2\rar\R/\Z$ for the restriction of $L_i$ to $\Z^2$, modulo $1$, and notice that $\mc{L}_1+\mc{L}_2+\mc{L}_3=0$. This means that the orbit of $0$ under the natural $\Z^2$-action of $E$ on $F_\rho/\Z^3$ is contained in the two dimensional rational subtorus with equation $x+y+z=0$. The kernels of the forms $\mc{L}_i$ are all rank $1$ subgroups of $\Z^2$, and it follows that the number of connected components of $\mathrm{reg}(r)$ which intersect the rational subtorus is $\asymp r^2$.

Since the forms are linearly dependent, we can understand the orbit of a point in $F_\rho/\Z^3$ under the $\Z^2$-action by considering only the values of $\mc{L}_1$ and $\mc{L}_3$. In other words, we can consider the projection of the problem onto the $e_1e_3$-plane. Consider the intersection of a connected component of $\mathrm{reg}(r)$ with the subspace $x+y+z=0$. This is a two dimensional region, bounded by the intersections of the subspace with translates (by the $\Z^5$ action) of the hyperplanes forming the boundary of the canonical window. Computing the vertices of the region is an operation which takes place in $\Q(\sqrt{5})$. As in the previous example, this leads to the conclusion that the intersection of any connected component of $\mathrm{reg}(r)$ with the subspace $x+y+z=0$, when projected to the $e_1e_3$-plane, contains a square of side length $\gg r^{-1}.$ The remainder of the proof follows exactly as before, allowing us to conclude that $Y$ is LR.

\vspace{.2in}

{\footnotesize

\noindent AH\,:
University of Houston, Texas, USA\\
haynes@math.uh.edu

\vspace{.1in}

\noindent HK\,:
University of Vienna, Austria\\
henna.koivusalo@univie.ac.at

\vspace{.1in}

\noindent JW\,:
University of Durham, UK\\
james.j.walton@durham.ac.uk

}

\end{document}